\documentclass[12pt,reqno]{amsart}
\usepackage[utf8]{inputenc}
\usepackage[T1]{fontenc}
\usepackage{amsmath}
\usepackage{amsfonts}
\usepackage{amssymb,amsthm}
\usepackage{graphicx}
\usepackage{booktabs}
\usepackage{float}
\usepackage{cancel}
\usepackage{float}
\usepackage{mathrsfs}
\usepackage{hyperref}
\usepackage{color}
\usepackage{algorithm}
\usepackage{algpseudocode}
\usepackage{etoolbox}
\usepackage{comment}
\usepackage[margin=1in]{geometry}
\usepackage{tikz}
\usepackage{pgfplots}
\usetikzlibrary{angles, quotes}
\newtheorem{thm}{Theorem}[section]
\newtheorem{defi}[thm]{Definition}
\newtheorem{pro}[thm]{Proposition}
\newtheorem{cor}[thm]{Corollary}
\newtheorem{lem}[thm]{Lemma}
\newtheorem{rem}[thm]{Remark}

\title[Non-Uniform Discrete Pr\"ufer]{Theory and Computation of Discrete Sturm--Liouville Problems on Non-Uniform Grids via the Pr\"ufer Transformation}
\author{Bailyn Hall\textsuperscript{1}, Kimsear Lor\textsuperscript{2}\\ \\ Project Advisors: Shalmali Bandyopadhyay\textsuperscript{3}, Jacob Blazejewski\textsuperscript{4}}
\thanks{\textsuperscript{1}Department of Mathematics and Statistics, University of Tennessee at Martin, Martin, TN. ORCID: 0009-0000-9908-1661}
\thanks{\textsuperscript{2}Department of Mathematics and Statistics, University of Tennessee at Martin, Martin, TN. ORCID: 0009-0000-3564-260X }
\thanks{\textsuperscript{3}Department of Mathematics and Statistics, University of Tennessee at Martin, Martin, TN. ORCID: 0000-0003-3971-6033}
\thanks{\textsuperscript{4}Department of Mathematical Sciences, Appalachian State University, Boone, NC. ORCID: 0000-0002-5346-3207}
\usepackage{amsmath, amssymb, amsthm}

\usepackage{graphicx}
\usepackage{tikz}
\usepackage{hyperref}
\newtheorem{definition}{Definition}[section]
\newcommand{\R}{\mathbb{R}}
\newcommand{\N}{\mathbb{N}}

\numberwithin{equation}{section}

\usepackage{multicol}

\begin{document}

\begin{abstract}
We study a discrete version of the Sturm--Liouville eigenvalue problem on grids whose spacing may vary from point to point, using the discrete Pr\"ufer transformation. We show that any eigenvalues of the problem are real and that there are finitely many of them. We then compare two numerical methods for computing the eigenvalues, regular shooting and Pr\"ufer-based shooting, and find that the Pr\"ufer method remains accurate on non-uniform grids where regular shooting loses accuracy or fails.
\end{abstract}

\maketitle
\markright{\MakeUppercase{Non-Uniform Discrete Pr\"ufer}}
\pagestyle{headings}
\noindent {\bf Keywords:} Non-uniform grid, difference equation, Sturm-Liouville Theory, Pr\"ufer Transformation

\smallskip

\noindent {\bf MSc Subject Classificatin (2020):} 39A27, 39A06, 65Q10

\section{Introduction}
\label{sec:intro}

We consider the Sturm--Liouville eigenvalue problem with Dirichlet boundary conditions,
\begin{equation} \label{eq:main}
\begin{cases}
\Delta_h\!\left(r_k\Delta_h x_k\right) + \left(\lambda w_k - q_k\right) x_{k+1} = 0, & k = 0, \ldots, n-1, \\
x_0 = x_{n+1} = 0,
\end{cases}
\end{equation}
where \(r_k, w_k, q_k : \N \to \R\) are real sequences with \(r_k \neq 0\) and \(w_k > 0\) for all \(k\), and \(\Delta_h\) is the non-uniform difference operator
\begin{equation}
\label{def:2.2}
\Delta_h f_k = \frac{f_{k+1} - f_k}{t_{k+1} - t_k} = \frac{f_{k+1} - f_k}{h_k},    
\end{equation}
defined on a grid \(\{t_k\}\) with step sizes \(h_k = t_{k+1} - t_k > 0\) that may vary from node to node. This is a second-order difference equation, and we wish to find the values of the parameter \(\lambda\), called eigenvalues, for which it has a nontrivial solution.

Equation \eqref{eq:main} is the discrete version of the Sturm--Liouville differential equation
\[
(r(x)y')' + \left(\lambda w(x) - q(x)\right) y = 0.
\]
Such equations appear throughout mathematical physics, since separating variables in the heat, wave, and related equations leads to this form. What sets a Sturm--Liouville equation apart from a general second-order equation is the shape of its leading term: it is a derivative of a coefficient times a derivative, $(r(x)y')'$, rather than a plain second derivative. The coefficient $r(x)$ is a physical quantity, such as the varying stiffness or conductivity of a material, so the product $r(x) y'$ is the meaningful quantity to track rather than the slope $y'$ by itself. In the discrete setting this quantity is $r_k \Delta_h x_k$, which we call the \emph{quasi-derivative}, and the leading term $\Delta_h(r_k \Delta_h x_k)$ is a difference of the quasi-derivative, mirroring $(r(x)y')'$ in the continuous case. Interest in discrete Sturm--Liouville problems of this kind has grown more recently, alongside the development of numerical methods for computing their eigenvalues.

To study \eqref{eq:main} we use the discrete Pr\"ufer transformation, which replaces the single second-order equation by a pair of first-order equations for an amplitude and a phase, from which the eigenvalues are easier to extract. The transformation is built on the quasi-derivative $r_k \Delta_h x_k$ rather than on the plain difference $\Delta_h x_k$. This choice is what keeps the transformation aligned with the structure of the equation: because the quasi-derivative is the natural quantity in the self-adjoint form, taking $x_k$ and $r_k \Delta_h x_k$ as the two Pr\"ufer coordinates preserves the self-adjoint structure and keeps the resulting phase equation well posed.

Bohner--Do\v{s}l\'y~\cite{bohner-dosly} introduced the quasi-derivative Pr\"ufer transformation for self-adjoint difference equations, and \c{C}etinkaya--Er--Menken~\cite{cetinkaya-er-menken} used it to estimate eigenvalues, both on uniform grids. Our contribution is to carry the quasi-derivative Pr\"ufer transformation onto non-uniform grids and to study the resulting problem. We prove two main results. The first shows that a nontrivial solution of \eqref{eq:main} corresponds exactly to a solution of the Pr\"ufer system, so the two formulations are equivalent and the phase is well defined. The second shows that any eigenvalue of \eqref{eq:main} is real, and that the problem has exactly $n$ eigenvalues, counted with multiplicity.

Before stating these results, we record the Pr\"ufer substitution itself.
\begin{definition}
    The \emph{Pr\"ufer substitution} for \eqref{eq:main} is given by
    \begin{equation}
x_k=\mu_k\sin\theta_k,\label{lambda eq:Prufer A}
    \end{equation} 
    and
    \begin{equation}
    r_k\Delta_hx_k=\mu_k\cos\theta_k, \label{lambda eq:Prufer B} 
    \end{equation}
    where $\mu_k>0$ and $\theta_k \in [0, 2 \pi)$, for $k \in \{0,1, \ldots,n\}$.
\end{definition}

The first main result states the equivalence between the original equation and the Pr\"ufer system.
\begin{thm}\label{thm:prufer}
Fix $\lambda \in \mathbb{R}$, and suppose $\mu_k > 0$ and $\theta_k$ are related to
$x_k$ by \eqref{lambda eq:Prufer A} and \eqref{lambda eq:Prufer B}. Then $x$ is a
nontrivial solution of \eqref{eq:main} if and only if $\mu_k$ and $\theta_k$ satisfy
\begin{align}
\Delta_h\mu_k = -\mu_k\Bigg(\frac{h_k\bigl[(\Delta_h\sin\theta_k)^2 + (\Delta_h\cos\theta_k)^2\bigr]}{2}-\frac{1}{r_k}\sin\theta_{k+1}\cos\theta_k\notag\\
+
\frac{\left(\lambda w_k-q_k\right)\, h_k}{r_k}\cos\theta_{k+1}\cos\theta_{k}
&+ \left(\lambda w_k-q_k\right)\cos\theta_{k+1}\sin\theta_k\Bigg) \label{lambda eq:thmA}
\end{align}
\begin{align}
\label{lambda eq:thmB}
\sin\bigl(h_k\,\Delta_h\theta_k\bigr)
={}& \frac{h_k}{r_k}\Bigl(\cos\theta_k\cos\theta_{k+1}
+ \left(\lambda w_k-q_k\right)\,h_k\cos\theta_k\sin\theta_{k+1}+ \left(\lambda w_k-q_k\right)\,r_k\sin\theta_k\sin\theta_{k+1}\Bigr)
\end{align}
\end{thm}

The second main result concerns the eigenvalues themselves.
\begin{thm}\label{thm:spectral}
If $\lambda$ is an eigenvalue of \eqref{eq:main}, then $\lambda$ is real. Moreover, \eqref{eq:main} has exactly $n$ eigenvalues, counted with multiplicity.
\end{thm}

The rest of the paper is organized as follows. Section~\ref{sec:prelim} develops the discrete-calculus tools and the operator framework, and shows that the operator behind \eqref{eq:main} is self-adjoint. Section~\ref{sec:Prufer} proves Theorem~\ref{thm:prufer}, and Section~\ref{sec:spectral} proves Theorem~\ref{thm:spectral}. Section~\ref{sec:numerical} compares two numerical methods for computing the eigenvalues, regular shooting and Pr\"ufer-based shooting, on three families of grids.
\section{Preliminaries}
\label{sec:prelim}
In this section we collect the discrete-calculus tools needed in the rest of the paper and use them to build the operator framework behind the eigenvalue problem \eqref{eq:main}. The eventual goal is to show that the operator associated with \eqref{eq:main} is self-adjoint with respect to a suitable inner product, since self-adjointness is what forces the eigenvalues to be real in Section~\ref{sec:spectral}. We begin with two known identities from discrete calculus, the product rule and a summation-by-parts formula, taken from \cite{Book KP}; these are the algebraic backbone of everything that follows.

\begin{pro}
Let \(f_k,g_k: \N \to \R\), then 
\begin{equation}
\label{pro: product rule}
\displaystyle \Delta_h \left(f_k g_k\right)=g_{k+1} \Delta_h f_k+f_k\Delta_h g_k  = f_{k+1} \Delta_h g_k+g_k\Delta_h f_k  = \displaystyle \Delta_h \left(g_k f_k\right)
\end{equation}
\end{pro}

\begin{proof}
For a detailed proof, we refer the reader to \cite[Pg. 15, Thm 2.1]{Book KP}.
\end{proof}

The next proposition is the discrete analogue of the fundamental theorem of calculus for the operator $\Delta_h$, and it supplies the boundary terms that later vanish under the Dirichlet conditions.

\begin{pro}
\label{pro:summation}
Let $\{a_k\}$ and $\{b_k\}$ be sequences of real numbers and $h_k$ defined as in \eqref{def:2.2}, then
\[
\sum_{k=0}^{n}h_k\Delta_h(a_kb_k) = a_{n+1}b_{n+1} - a_0b_0.
\]
\end{pro}
\begin{proof}
For a detailed proof, we refer the reader to \cite[Pg. 24, Thm 2.6(c,d)]{Book KP}.
\end{proof}

Combining the product rule with Proposition~\ref{pro:summation} yields the summation-by-parts identity below, which is the discrete counterpart of integration by parts and the key step in the self-adjointness proof.

\begin{cor}
\label{cor:summation}
Let $\{a_k\}$ and $\{b_k\}$ be sequences of real numbers, with $h_k$ defined as in \eqref{def:2.2}. If $b_0 = b_{n+1} = 0$, then
\[
\sum_{k=0}^{n} h_k\, b_{k+1}\, \Delta_h a_k = -\sum_{k=0}^{n} h_k\, a_k\, \Delta_h b_k.
\]
\end{cor}

\begin{proof}
By Proposition \ref{pro: product rule}, $\Delta_h(a_k b_k) = a_k\, \Delta_h b_k + b_{k+1}\, \Delta_h a_k.$ Multiplying both sides by $h_k$ and summing from $k = 0$ to $k=n$, we get $\displaystyle \sum_{k=0}^{n} h_k\, \Delta_h(a_k b_k) = \displaystyle \sum_{k=0}^{n} h_k\, a_k\, \Delta_h b_k + \sum_{k=0}^{n} h_k\, b_{k+1}\, \Delta_h a_k.$
By Proposition \ref{pro:summation} and using the fact $b_0 = b_{n+1} = 0$, we have $\displaystyle\sum_{k=0}^{n} h_k\, \Delta_h(a_k b_k) = a_{n+1}b_{n+1} - a_0 b_0=0$, yielding
\[
\sum_{k=0}^{n} h_k\, b_{k+1}\, \Delta_h a_k = -\sum_{k=0}^{n} h_k\, a_k\, \Delta_h b_k,
\]
as claimed.
\end{proof}

With these identities in hand, we introduce the inner product on sequences with respect to which self-adjointness is measured.

\begin{defi}
Let $u=\left\{u_k\right\}_{k=0}^{n+1}$ and $v=\left\{v_k\right\}_{k=0}^{n+1}$. We define the \emph{inner product} of $u$ and $v$ as
\begin{equation}
\label{eq:inner-product}
\left\langle u,v\right\rangle := \sum_{k=0}^{n} h_k\, u_{k+1} \overline{v_{k+1}}.
\end{equation}
\end{defi}

The following proposition records that \eqref{eq:inner-product} is a genuine inner product, so that the notions of adjoint and self-adjointness are well defined.

\begin{pro}
The inner product defined by \eqref{eq:inner-product} has the following properties.
\begin{description}
    \item[Symmetry] $\left\langle x,y\right\rangle=\overline{\left\langle y,x\right\rangle}$.
    \item[Linearity in the first argument] For all \(\alpha,\beta\in\mathbb{C}\), $\left\langle\alpha x+\beta z, y\right\rangle=\alpha\left\langle x,y\right\rangle+\beta\left\langle z,y\right\rangle$.
    \item[Positive Definiteness] $\left\langle x,x\right\rangle\geq0$, and $\left\langle x,x\right\rangle= 0$ if and only if $x_{k+1}=0$ for all $k \in \{0,1,\dots,n\}$.
\end{description}
\end{pro}

These properties are easily verifiable; we refer the reader to \cite[Pg. 282, Def. 7.4]{Book KP}. We now introduce the \emph{principal operator} associated with \eqref{eq:main}, which carries the second-order part of the equation, and prove that it is self-adjoint.

\begin{defi}
The \emph{principal operator} associated with \eqref{eq:main} is the linear operator $\mathcal{L}_0$ acting on sequences $\{x_k\}_{k=0}^{n+1}$ satisfying the Dirichlet boundary conditions $x_0 = x_{n+1} = 0$, defined by
\begin{equation}
    \label{def:op}
    (\mathcal{L}_0 x)_{k+1} = \Delta_h(r_k \Delta_h x_k).
\end{equation}
\end{defi}
\begin{lem}
\label{lem:opL0}
The principal operator $\mathcal{L}_0$ defined in \eqref{def:op} is self-adjoint with respect to the inner product \eqref{eq:inner-product}; that is,
\[
\langle \mathcal{L}_0 x, y \rangle = \langle x, \mathcal{L}_0 y \rangle
\]
for all sequences $x, y$ satisfying the Dirichlet boundary conditions.
\end{lem}
\begin{proof}
Observe that $\langle \mathcal{L}_0x,y \rangle = \displaystyle\sum_{k=0}^{n}h_k\Delta_h(r_k \Delta_h x_{k})\overline{y_{k+1}}$ and $\langle x,\mathcal{L}_0y \rangle = \displaystyle\sum_{k=0}^{n}h_k\Delta_h(r_k \Delta_h y_{k})\overline{x_{k+1}}$. Let 
\begin{align*}
a_k &= r_k\Delta_h x_k,  \qquad \widetilde{a}_k = r_k\Delta_h y_k,\\
b_k &= \overline{y_k}, \qquad \widetilde{b}_k=\overline{x_k}.
\end{align*}
By employing Corollary \ref{cor:summation}, we obtain
\begin{align*}
    \displaystyle\sum_{k=0}^{n}h_k\Delta_h(r_k\Delta_hx_k)\overline{y_{k+1}} &= -\displaystyle\sum_{k=0}^{n}h_k(r_k\Delta_hx_k)\Delta_h\overline{y_k},\\
     \displaystyle\sum_{k=0}^{n}h_k\Delta_h(r_k\Delta_hy_k)\overline{x_{k+1}} &= -\displaystyle\sum_{k=0}^{n}h_k(r_k\Delta_hy_k)\Delta_h\overline{x_k},
\end{align*}
which imply $\displaystyle\sum_{k=0}^{n}h_k\Delta_h(r_k \Delta_h x_{k})\overline{y_{k+1}}-\displaystyle\sum_{k=0}^{n}h_k\Delta_h(r_k \Delta_h y_{k})\overline{x_{k+1}}=0$, using that $r_k$ is real. Therefore, we conclude $\langle \mathcal{L}_0x,y \rangle= \langle x,\mathcal{L}_0y \rangle$.
\end{proof}

The lower-order potential term is collected into a multiplication operator, after which the full operator behind \eqref{eq:main} is assembled and shown to inherit self-adjointness.

\begin{lem}
\label{lem:opA}
Let $(Qx)_{k+1} := q_kx_{k+1}$, and define $(\mathcal{A}x)_{k+1} = (Qx)_{k+1}-(\mathcal{L}_0x)_{k+1}$. Then $\mathcal{A}$ is self-adjoint.
\end{lem}
\begin{proof}
Observe that, since $q_k$ is real,
\[
\langle Qx, y \rangle = \sum_{k=0}^{n} h_k(q_kx_{k+1})\overline{y_{k+1}} = \sum_{k=0}^{n} h_k x_{k+1}\overline{(q_ky_{k+1})}= \langle x, Qy\rangle.
\]
By Lemma \ref{lem:opL0}, we have $\langle \mathcal{L}_0x,y\rangle = \langle x, \mathcal{L}_0y\rangle$. Therefore,
\begin{align*}
        \langle \mathcal{A}x, y\rangle= \langle (-\mathcal{L}_0 + Q)x, y\rangle= \langle -\mathcal{L}_0x, y\rangle + \langle Qx, y \rangle = \langle x, -\mathcal{L}_0y\rangle + \langle x, Qy\rangle &= \langle x, (-\mathcal{L}_0+Q)y\rangle\notag \\
        &=\langle x, \mathcal{A}y\rangle,
    \end{align*}
i.e., $\mathcal{A}$ is self-adjoint.
\end{proof}


\section{Proof of Theorem \ref{thm:prufer}}
\label{sec:Prufer}
In this section, we prove Theorem \ref{thm:prufer}, establishing the equivalence between nontrivial solutions of \eqref{eq:main} and solutions of the Pr\"ufer system.

We first prove the forward direction. Suppose $x$ is a nontrivial solution of \eqref{eq:main}. We begin by applying Proposition \ref{pro: product rule} to \eqref{lambda eq:Prufer A} which yields
\begin{equation}
\Delta_h x_k
= \sin\theta_{k+1}\,\Delta_h\mu_k + \mu_k\,\Delta_h\sin\theta_k, \notag
\end{equation}
and using \eqref{lambda eq:Prufer B}, we obtain
\begin{equation}
\frac{\mu_k\cos\theta_k}{r_k}
= \sin\theta_{k+1}\,\Delta_h\mu_k + \mu_k\,\Delta_h\sin\theta_k .\label{lambda eqA}
\end{equation}
Next, applying Proposition \ref{pro: product rule} to \eqref{lambda eq:Prufer B} yields, 
\begin{equation}
\Delta_h\bigl(r_k\,\Delta_h x_k\bigr)
= \cos\theta_{k+1}\,\Delta_h\mu_k + \mu_k\,\Delta_h\cos\theta_k . \notag
\end{equation}
Using \eqref{eq:main}, we obtain,
\begin{equation}
-\frac{\left(\lambda w_k-q_k\right)\, h_k\, \mu_k}{r_k}\cos\theta_k
- \left(\lambda w_k-q_k\right)\,\mu_k\sin\theta_k
= \cos\theta_{k+1}\,\Delta_h\mu_k + \mu_k\,\Delta_h\cos\theta_k .\label{lambda eqB}
\end{equation}
Multiply
\eqref{lambda eqA} by $\sin\theta_{k+1}$ and \eqref{lambda eqB} by $\cos\theta_{k+1}$, which yields
\begin{align}
(\sin\theta_{k+1})\bigg(\frac{\mu_k\cos\theta_k}{r_k}
&= \sin\theta_{k+1}\,\Delta_h\mu_k + \mu_k\,\Delta_h\sin\theta_k\bigg) \notag \\ 
\Rightarrow \sin^2\theta_{k+1}\,\Delta_h\mu_k
+ \mu_k\sin\theta_{k+1}\,\Delta_h\sin\theta_k
&= \frac{\mu_k}{r_k}\sin\theta_{k+1}\cos\theta_k,
\label{lambda eq:3.8}
\end{align}
\begin{align}
(\cos\theta_{k+1})\bigg(-\frac{\left(\lambda w_k-q_k\right)\, h_k}{r_k}\,\mu_k\cos\theta_k
- \left(\lambda w_k-q_k\right)\,\mu_k\sin\theta_k
&= \cos\theta_{k+1}\,\Delta_h\mu_k + \mu_k\,\Delta_h\cos\theta_k\bigg) \notag \\
\Rightarrow \cos^2\theta_{k+1}\,\Delta_h\mu_k
+ \mu_k\cos\theta_{k+1}\,\Delta_h\cos\theta_k
&= -\frac{\left(\lambda w_k-q_k\right)\, h_k}{r_k}\,\mu_k\cos\theta_{k+1}\cos\theta_k \notag
\\ &\hspace{0.75cm}- \left(\lambda w_k-q_k\right)\,\mu_k\cos\theta_{k+1}\sin\theta_k .
\label{lambda eq:3.9}
\end{align}

Adding \eqref{lambda eq:3.8} and \eqref{lambda eq:3.9},
\begin{align*}
\text{LHS}
&= \sin^2\theta_{k+1}\,\Delta_h\mu_k
+ \mu_k\sin\theta_{k+1}\,\Delta_h\sin\theta_k
+\cos^2\theta_{k+1}\,\Delta_h\mu_k
+ \mu_k\cos\theta_{k+1}\,\Delta_h\cos\theta_k\\
&= \Delta_h\mu_k\bigl(\sin^2\theta_{k+1} + \cos^2\theta_{k+1}\bigr)
+ \mu_k\bigl(\sin\theta_{k+1}\,\Delta_h\sin\theta_k
+ \cos\theta_{k+1}\,\Delta_h\cos\theta_k\bigr) \notag\\
&= \Delta_h\mu_k
+ \mu_k\bigl(\sin\theta_{k+1}\,\Delta_h\sin\theta_k
+ \cos\theta_{k+1}\,\Delta_h\cos\theta_k\bigr).
\end{align*}
Observe that,
\begin{equation}
\label{lambda eq:trig-identity}
\sin\theta_{k+1}\,\Delta_h\sin\theta_k
+ \cos\theta_{k+1}\,\Delta_h\cos\theta_k
= \frac{h_k\bigl[(\Delta_h\sin\theta_k)^2 + (\Delta_h\cos\theta_k)^2\bigr]}{2}
\end{equation}
Therefore,
\begin{align}
\label{lambda eq:LHS-phase}
\text{LHS} &= \Delta_h\mu_k + \mu_k\Bigg(\frac{h_k\bigl[(\Delta_h\sin\theta_k)^2 + (\Delta_h\cos\theta_k)^2\bigr]}{2}\Bigg)
\end{align}
\begin{align}
\label{lambda eq:RHS-phase}
\text{RHS} &= \frac{\mu_k}{r_k}\sin\theta_{k+1}\cos\theta_k-\frac{\left(\lambda w_k-q_k\right)\, h_k\, \mu_k}{r_k}\cos\theta_{k+1}\cos\theta_{k}
- \left(\lambda w_k-q_k\right)\,\mu_k\cos\theta_{k+1}\sin\theta_k.
\end{align}
Combining \eqref{lambda eq:RHS-phase} and \eqref{lambda eq:LHS-phase} together and isolating $\Delta_h\mu_k$, we obtain 
\begin{align}
\Delta_h\mu_k = -\mu_k\Bigg(\frac{h_k\bigl[(\Delta_h\sin\theta_k)^2 + (\Delta_h\cos\theta_k)^2\bigr]}{2}\Bigg)+\mu_k\left(\frac{1}{r_k}\sin\theta_{k+1}\cos\theta_k\right)\notag\\
+
\mu_k\left(\frac{-\left(\lambda w_k-q_k\right)\, h_k}{r_k}\cos\theta_{k+1}\cos\theta_{k}\right) \notag \\
+\mu_k\left(- \left(\lambda w_k-q_k\right)\cos\theta_{k+1}\sin\theta_k\right),
\end{align}
which in turn yields \eqref{lambda eq:thmA}, as desired.

Next, to obtain \eqref{lambda eq:thmB}, we multiply \eqref{lambda eqA} by $\cos\theta_{k+1}$ and
\eqref{lambda eqB} by $-\sin\theta_{k+1}$ and add.
\begin{align}
    (\cos\theta_{k+1})\bigg(\frac{\mu_k\cos\theta_k}{r_k}
&= \sin\theta_{k+1}\,\Delta_h\mu_k + \mu_k\,\Delta_h\sin\theta_k\bigg)\\
\Rightarrow \frac{\mu_k}{r_k}\cos\theta_{k+1}\cos\theta_k
&= \cos\theta_{k+1}\sin\theta_{k+1}\,\Delta_h\mu_k
+ \mu_k\cos\theta_{k+1}\,\Delta_h\sin\theta_k .
\label{lambda eq:3.16}
\end{align}

\begin{align}
    (-\sin\theta_{k+1})\bigg(-\frac{\left(\lambda w_k-q_k\right)\, h_k}{r_k}\,\mu_k\cos\theta_k
- \left(\lambda w_k-q_k\right)\,\mu_k\sin\theta_k
&= \cos\theta_{k+1}\,\Delta_h\mu_k + \mu_k\,\Delta_h\cos\theta_k\bigg)\\
\Rightarrow -\sin\theta_{k+1}\cos\theta_{k+1}\,\Delta_h\mu_k
- \mu_k\sin\theta_{k+1}\,\Delta_h\cos\theta_k
&=h_k\,\frac{\left(\lambda w_k-q_k\right)}{r_k}\,\mu_k\cos\theta_k\sin\theta_{k+1} \notag
\\ &\hspace{0.75cm} + \left(\lambda w_k-q_k\right)\,\mu_k\sin\theta_{k+1}\sin\theta_k .
\label{lambda eq:3.17}
\end{align}
Adding \eqref{lambda eq:3.16} and \eqref{lambda eq:3.17}, we obtain, 
\begin{align}
\text{LHS} &= \cos\theta_{k+1}\sin\theta_{k+1}\,\Delta_h\mu_k
+ \mu_k\cos\theta_{k+1}\,\Delta_h\sin\theta_k
- \sin\theta_{k+1}\cos\theta_{k+1}\,\Delta_h\mu_k
- \mu_k\sin\theta_{k+1}\,\Delta_h\cos\theta_k \notag\\
&= \Delta_h\mu_k\bigl(\cos\theta_{k+1}\sin\theta_{k+1}
- \sin\theta_{k+1}\cos\theta_{k+1}\bigr) -\mu_k\bigl(-\cos\theta_{k+1}\,\Delta_h\sin\theta_k
+ \sin\theta_{k+1}\,\Delta_h\cos\theta_k\bigr) \notag\\ 
&= 0 - \mu_k\bigg(-\frac{\cos\theta_{k+1}(\sin\theta_{k+1} - \sin\theta_k)}{h_k}+\frac{\sin\theta_{k+1}(\cos\theta_{k+1}-\cos\theta_k)}{h_k}\bigg) \notag  \\
&=\frac{-\mu_k}{h_k}(-\cos\theta_{k+1}\sin\theta_{k+1} + \cos\theta_{k+1}\sin\theta_k + \sin\theta_{k+1}\cos\theta_{k+1} - \sin\theta_{k+1}\cos\theta_k) \notag\\
&=\frac{\mu_k }{h_k}\sin(h_k\Delta_h\theta_k),
\label{lambda eq:LHS-amp}
\end{align}
and
\begin{align}
\text{RHS}&=\frac{\mu_k}{r_k}\cos\theta_k\cos\theta_{k+1}
+ \frac{h_k\,\left(\lambda w_k-q_k\right)}{r_k}\,\mu_k\cos\theta_k\sin\theta_{k+1}
+ \left(\lambda w_k-q_k\right)\,\mu_k\sin\theta_k\sin\theta_{k+1} \notag\\
&=\mu_k\bigg(\frac{1}{r_k}\cos\theta_k\cos\theta_{k+1}
+ \frac{h_k\,\left(\lambda w_k-q_k\right)}{r_k}\,\cos\theta_k\sin\theta_{k+1}
+ \left(\lambda w_k-q_k\right)\,\sin\theta_k\sin\theta_{k+1}\bigg)
\label{lambda eq:RHS-amp}
\end{align}
Combining \eqref{lambda eq:LHS-amp} and \eqref{lambda eq:RHS-amp} we get,
\begin{equation*}
\frac{\mu_k}{h_k}\sin\bigl(h_k\,\Delta_h\theta_k\bigr)
=\frac{\mu_k}{r_k}\Bigl(\cos\theta_k\cos\theta_{k+1}
+ h_k\,\left(\lambda w_k-q_k\right)\cos\theta_k\sin\theta_{k+1}
+ \left(\lambda w_k-q_k\right)r_k\sin\theta_k\sin\theta_{k+1}\Bigr)
\end{equation*}
which in turn yields \eqref{lambda eq:thmB}. This establishes the forward direction.

Observe that, the two multiplications performed above can be recorded compactly. The linear combinations that produce \eqref{lambda eq:thmA} and \eqref{lambda eq:thmB} from \eqref{lambda eqA} and \eqref{lambda eqB} are
\[
\begin{pmatrix} \eqref{lambda eq:thmA} \\ \eqref{lambda eq:thmB} \end{pmatrix} = M \begin{pmatrix} \eqref{lambda eqA} \\ \eqref{lambda eqB}\end{pmatrix}, \qquad M = \begin{pmatrix} \sin\theta_{k+1} & \cos\theta_{k+1} \\ \cos\theta_{k+1} & -\sin\theta_{k+1} \end{pmatrix}.
\]
Since $M^{-1} = M$, the transformation is an involution, and therefore
\[
 \begin{pmatrix} \eqref{lambda eqA} \\ \eqref{lambda eqB}\end{pmatrix}= M \begin{pmatrix} \eqref{lambda eq:thmA} \\ \eqref{lambda eq:thmB} \end{pmatrix}.
\]

We now prove the converse. Suppose $\mu_k$ and $\theta_k$ satisfy \eqref{lambda eq:thmA} and \eqref{lambda eq:thmB}. Notice that \eqref{lambda eq:thmA} can be rewritten as
\begin{align}
\label{lambda eq:re-thmA}
\Delta_h \mu_k + \mu_k \bigl[ \sin\theta_{k+1} \Delta_h \sin\theta_k + \cos\theta_{k+1} \Delta_h \cos\theta_k \bigr] \notag \\ = \frac{\mu_k}{r_k} \sin\theta_{k+1} \cos\theta_k - \frac{\left(\lambda w_k-q_k\right) h_k \mu_k}{r_k} \cos\theta_{k+1} \cos\theta_k &- \left(\lambda w_k-q_k\right) \mu_k \cos\theta_{k+1} \sin\theta_k,
\end{align}
and \eqref{lambda eq:thmB} can be rewritten as 
\begin{align}
\label{lambda eq:re-thmB}
\mu_k \bigl[ \cos\theta_{k+1} \Delta_h \sin\theta_k - \sin\theta_{k+1} \Delta_h \cos\theta_k \bigr] \notag \\= \frac{\mu_k}{r_k} \cos\theta_{k+1} \cos\theta_k + &\frac{\left(\lambda w_k-q_k\right) h_k \mu_k}{r_k} \sin\theta_{k+1} \cos\theta_k + \left(\lambda w_k-q_k\right) \mu_k \sin\theta_{k+1} \sin\theta_k.
\end{align}
Therefore, via the involution identity above, we obtain  \(\eqref{lambda eqB}=\cos\theta_{k+1}\eqref{lambda eq:thmA}-\sin\theta_{k+1}\eqref{lambda eq:thmB}\).
For simplicity, we will consider the left and right hand sides separately.
\begin{align}
\notag \text{LHS}&=\cos\theta_{k+1}(\text{LHS of } \eqref{lambda eq:thmA})-\sin\theta_{k+1}\left(\text{LHS of }\eqref{lambda eq:thmB}\right)\notag\\
&=\cos\theta_{k+1}\Delta_h\mu_k+\mu_{k}\left(\cos\theta_{k+1}\sin\theta_{k+1}\Delta_h\sin\theta_k+\cos^2\theta_{k+1}\Delta_h\cos\theta_k\right) \notag \\ 
&\hspace{3.7cm}-\mu_k\left(\cos\theta_{k+1}\sin\theta_{k+1}\Delta_h\sin\theta_k-\sin^2\theta_{k+1}\Delta_h\cos\theta_k\right) \notag \\ 
&=\cos\theta_{k+1}\Delta_h\mu_k+\mu_k\cos^2\theta_{k+1}\Delta_h\cos\theta_k+\mu_k\sin^2\theta_{k+1}\Delta_h\cos\theta_k \notag \\ 
    & =\cos\theta_{k+1}\Delta_h\mu_k+\mu_k\Delta_h\cos\theta_k 
    \label{lambda eq:LHS-4.1}
\end{align}

\begin{align}
    \text{RHS}&=\cos\theta_{k+1}(\text{RHS of } \eqref{lambda eq:thmA})-\sin\theta_{k+1}\left(\text{RHS of }\eqref{lambda eq:thmB}\right)\notag\\
    &=\frac{\mu_k}{r_k}\sin\theta_{k+1}\cos\theta_{k+1}\cos\theta_k-\frac{\left(\lambda w_k-q_k\right)\,h_k\,\mu_k}{r_k}\cos^2\theta_{k+1}\cos\theta_k-\left(\lambda w_k-q_k\right)\mu_k\cos^2\theta_{k+1}\sin\theta_k\notag\\
    &\hspace{1cm}-\frac{\mu_k}{r_k}\cos\theta_{k+1}\sin\theta_{k+1}\cos\theta_k-\frac{\left(\lambda w_k-q_k\right)\,h_k\,\mu_k}{r_k}\sin^2\theta_{k+1}\cos\theta_k\notag -\left(\lambda w_k-q_k\right)\mu_k\sin^2\theta_{k+1}\sin\theta_k \\ \notag
    &=\frac{\mu_k}{r_k}\sin\theta_{k+1}\cos\theta_{k+1}\cos\theta_k-\frac{\mu_k}{r_k}\sin\theta_{k+1}\cos\theta_{k+1}\cos\theta_k \\ \notag &\hspace{1cm}-\frac{\left(\lambda w_k-q_k\right)\,h_k\,\mu_k}{r_k}\cos\theta_k(\cos^2\theta_{k+1}+\sin^2\theta_{k+1}) -\left(\lambda w_k-q_k\right)\mu_k\sin\theta_k(\sin^2\theta_{k+1}+\cos^2\theta_{k+1}) \notag\\ 
    &=-\frac{\left(\lambda w_k-q_k\right)\,h_k\,\mu_k}{r_k}\cos\theta_k-\left(\lambda w_k-q_k\right)\mu_k\sin\theta_k
    \label{lambda eq:RHS-4.1}
\end{align}
Equating \eqref{lambda eq:LHS-4.1} and \eqref{lambda eq:RHS-4.1}, we have,
\begin{equation}
\cos\theta_{k+1}\Delta_h\mu_k+\mu_k\Delta_h\cos\theta_k=\frac{-\left(\lambda w_k-q_k\right)h_k\mu_k}{r_k}\cos\theta_k-\left(\lambda w_k-q_k\right)\mu_k\sin\theta_k \label{lambda eq:4.24}
\end{equation}
which yields \eqref{lambda eqB}. Observe that the left hand side of \eqref{lambda eq:4.24} is \(\Delta_h\left(\mu_k\cos\theta_{k}\right)\) via Proposition \ref{pro: product rule}. Using \eqref{lambda eq:Prufer B}, we obtain 
\begin{equation}
\label{lambda eq:LHS-DE}
\cos\theta_{k+1}\Delta_h\mu_k+\mu_k\Delta_h\cos\theta_k  = \Delta_h\left(\mu_k\cos\theta_{k}\right) =  \Delta_h\left(r_k\Delta_h x_k\right)  
\end{equation}
Now we take the RHS of \eqref{lambda eq:4.24} and use \eqref{lambda eq:Prufer A} - \eqref{lambda eq:Prufer B} which yields 
\begin{align}
\label{lambda eq:RHS-DE}
    -\frac{\left(\lambda w_k-q_k\right)\,h_k\,\mu_k}{r_k}\cos\theta_k-\left(\lambda w_k-q_k\right)\mu_k\sin\theta_k &= -\left(\lambda w_k-q_k\right)\left(\frac{h_k\,\mu_k}{r_k}\cos\theta_k+\mu_k\sin\theta_k\right) \notag \\
    & = -\left(\lambda w_k-q_k\right)\left(\frac{h_k}{r_k}r_k\Delta_h x_k+x_k\right)\notag \\
    & = -\left(\lambda w_k-q_k\right) \left(h_k\left(\frac{x_{k+1}-x_k}{h_k}\right)+x_k\right) \notag \\
    & = -\left(\lambda w_k-q_k\right) x_{k+1}.
\end{align}
Combining  \eqref{lambda eq:LHS-DE} and \eqref{lambda eq:RHS-DE}, we obtain \eqref{eq:main}. This concludes the proof. \hfill \qed

\begin{rem}
Notice that \eqref{lambda eq:thmB} is independent of $\mu_k$, so it determines
$\{\theta_k\}$ on its own. Once $\{\theta_k\}$ is known, \eqref{lambda eq:thmA} is
readily solved for $\mu_k$. Moreover, \eqref{lambda eq:thmB} advances the phase, one
index at a time, so the initial value $\theta_0$ generates the entire sequence
$\{\theta_k\}$ recursively.
\end{rem}
\section{Proof of Theorem \ref{thm:spectral}}
\label{sec:spectral}

In this section, we prove Theorem \ref{thm:spectral}. We first establish that every eigenvalue of \eqref{eq:main} is real. We begin by
rewriting \eqref{eq:main} in the operator form using $\mathcal{A}, Q, W$ as follows.
\begin{equation}
\label{eq:opSL}
\Delta_h(r_k\Delta_hx_k) + (\lambda w_k - q_k)x_{k+1}=0\Rightarrow (\mathcal{A}x)_{k+1}= \lambda(Wx)_{k+1}
\end{equation}
Note that $\lambda$ is an eigenvalue of \eqref{eq:main} if and only if $\lambda$ is an eigenvalue of \eqref{eq:opSL}. Moreover, 
\begin{align*}
& \langle Wx, y \rangle = \sum_{k=0}^{n}h_k(w_kx_{k+1})\overline{y_{k+1}}=\sum_{k=0}^{n}h_kx_{k+1}\overline{(w_ky_{k+1})}= \langle x, Wy \rangle \notag \\
 \Rightarrow & \lambda\langle Wx, x\rangle = \langle \lambda Wx, x\rangle= \langle \mathcal{A}x, x\rangle\notag= \langle x, \mathcal{A}x\rangle= \langle x, \lambda Wx\rangle= \overline{\lambda}\langle x, Wx\rangle=\overline{\lambda}\langle Wx, x\rangle\notag\\
\Rightarrow &(\lambda-\overline{\lambda})\langle Wx, x\rangle =0.
\end{align*}
Assume for a contradiction that $\langle Wx, x\rangle=0$. By definition of inner product \eqref{eq:inner-product}, $\langle Wx, x\rangle =\displaystyle\sum_{k=0}^{n} h_kw_kx_{k+1}\overline{x_{k+1}}=\displaystyle\sum_{k=0}^{n} h_kw_k|x_{k+1}|^2$. Since $h_k,~w_k >0$, $\langle Wx, x\rangle=0 \Rightarrow x_{k+1}=0$ for all $k \in \{0,1,\ldots,n\}$, together with Dirichlet boundary condition we obtain $x \equiv 0$, which contradicts the fact $\lambda $ is an eigenvalue of \eqref{eq:opSL}. Therefore, $\lambda-\overline{\lambda}=0 \Rightarrow \lambda = \overline{\lambda}$, i.e. $\lambda$ is real.

We next establish that \eqref{eq:main} has exactly $n$ eigenvalues, counted with multiplicity. We suppose $y_k=r_k\Delta_hx_k.$ Thus,
    \begin{gather*}
        \Delta_h\left(r_k\Delta_h x_k\right)+\left(\lambda w_k-q_k\right)x_{k+1}=0 \\
        \frac{y_{k+1}-y_{k}}{h_k}+\left(\lambda w_k-q_k\right)x_{k+1}=0 \\ 
        y_{k+1}-y_k-h_kq_kx_{k+1}=-\lambda h_k w_kx_{k+1} \\ 
        r_{k+1}\left(\frac{x_{k+2}-x_{k+1}}{h_{k+1}}\right)-r_k\left(\frac{x_{k+1}-x_k}{h_k}\right)-h_kq_kx_{k+1}=-\lambda h_kw_kx_{k+1} \\
        r_{k+1}\left(\frac{x_{k+2}}{h_{k+1}}-\frac{x_{k+1}}{h_{k+1}}\right)-r_k\left(\frac{x_{k+1}}{h_k}-\frac{x_{k}}{h_k}\right)-h_kq_kx_{k+1}=-\lambda h_kw_kx_{k+1} \\ 
        \frac{r_{k+1}x_{k+2}}{h_{k+1}}-\frac{r_{k+1}x_{k+1}}{h_{k+1}}-\frac{r_kx_{k+1}}{h_k}+\frac{r_kx_k}{h_k}-h_kq_kx_{k+1}=-\lambda h_kw_kx_{k+1} \\ x_{k+2}\left(\frac{r_{k+1}}{h_{k+1}}\right)+x_{k+1}\left(-\frac{r_{k+1}}{h_{k+1}}-\frac{r_k}{h_k}-h_kq_k\right)+x_k\left(\frac{r_k}{h_k}\right)=-\lambda h_kw_kx_{k+1} \\
        x_{k+2}\left(-\frac{r_{k+1}}{h_{k+1}}\right)+x_{k+1}\left(\frac{r_{k+1}}{h_{k+1}}+\frac{r_k}{h_k}+h_kq_k\right)+x_k\left(-\frac{r_k}{h_k}\right)=\lambda\left(h_kw_k\right)x_{k+1}
    \end{gather*}
    \begin{align*}D= \begin{bmatrix}
  {{h_0}{w_0}}&0& \cdots &0 \\ 
  0&{{h_1}{w_1}}& \cdots &0 \\ 
   \vdots &0& \ddots & \vdots  \\ 
  0&0& \cdots &{{h_{n - 1}}{w_{n - 1}}}
\end{bmatrix} \end{align*}
Observe that \(\left(D\right)_{j,j}=h_{j-1}w_{j-1}\) and otherwise, the entry is zero.

\begin{align*}
B=
\begin{bmatrix}
\dfrac{r_1}{h_1}+\dfrac{r_0}{h_0}+h_0 q_0 & -\dfrac{r_1}{h_1} & 0 & \cdots & 0 \\[1.4em]
-\dfrac{r_1}{h_1} & \dfrac{r_2}{h_2}+\dfrac{r_1}{h_1}+h_1 q_1 & -\dfrac{r_2}{h_2} & \ddots & \vdots \\[1.4em]
0 & -\dfrac{r_2}{h_2} & \dfrac{r_3}{h_3}+\dfrac{r_2}{h_2}+h_2 q_2 & \ddots & 0 \\[1.4em]
\vdots & \ddots & \ddots & \ddots & -\dfrac{r_{n-1}}{h_{n-1}} \\[1.4em]
0 & \cdots & 0 & -\dfrac{r_{n-1}}{h_{n-1}} &  \dfrac{r_n}{h_n}+\dfrac{r_{n-1}}{h_{n-1}}+h_{n-1} q_{n-1}
\end{bmatrix}
\end{align*}

\begin{align*}
&  B - \lambda D = M =\\  & 
\left[\begin{smallmatrix}
    \frac{r_1}{h_1} + \frac{r_0}{h_0} + h_0(q_0 - \lambda w_0) & -\frac{r_1}{h_1} & 0 & \cdots & 0 \\ 
    -\frac{r_1}{h_1} & \frac{r_2}{h_2} + \frac{r_1}{h_1} + h_1(q_1 - \lambda w_1) & -\frac{r_2}{h_2} & \ddots & \vdots  \\ 
    0 & -\frac{r_2}{h_2} & \frac{r_3}{h_3} + \frac{r_2}{h_2} + h_2(q_2 - \lambda w_2) & \ddots & 0 \\ 
    \vdots & \ddots & \ddots & \ddots & -\frac{r_{n-1}}{h_{n-1}} \\ 
    0 & \cdots & 0 & -\frac{r_{n-1}}{h_{n-1}} & \frac{r_n}{h_n} + \frac{r_{n-1}}{h_{n-1}} + h_{n-1}(q_{n-1} - \lambda w_{n-1}) 
  \end{smallmatrix}\right] \\
  \end{align*}
  The determinant of $M_{n\times n}$ is given by
\[
  \det(M_{n\times n}) = \sum_{\sigma \in S_n} \operatorname{sgn}(\sigma) \prod_{i=1}^n M_{i,\sigma(i)},
\]
where $S_n$ is the set of all permutations of $\{1,\dots,n\}$. Writing $M = B - \lambda D$, the entries satisfy
\[
  M_{i,\sigma(i)} =
  \begin{cases}
    \text{independent of } \lambda, & i \neq \sigma(i), \\[2pt]
    b_{ii} - \lambda d_{ii} \text{ with } d_{ii} \neq 0, & i = \sigma(i).
  \end{cases}
\]
Consequently, for a fixed $\sigma$, the product $\displaystyle \prod_{i=1}^n M_{i,\sigma(i)}$ is a polynomial in $\lambda$ whose degree equals the number of fixed points of $\sigma$, since only the diagonal factors (those with $i = \sigma(i)$) contribute a power of $\lambda$.
The identity permutation is the unique element of $S_n$ with $n$ fixed points; every other permutation has at most $n-2$, because no permutation can fix exactly $n-1$ points. Hence the term $\lambda^n$ arises solely from the identity, and its coefficient is
\[
  C_n = \prod_{i=1}^n (-d_{ii}) = (-1)^n \prod_{i=1}^n d_{ii} \neq 0.
\]
Therefore
\[
  \det M = \det(B - \lambda D) = P(\lambda),
  \qquad
  P(\lambda) = C_n\lambda^n + C_{n-1}\lambda^{n-1} + \dots + C_0,
\]
is a polynomial in $\lambda$ of degree exactly $n$. By the Fundamental Theorem of Algebra, $P$ has exactly $n$ roots in $\mathbb{C}$, counted with multiplicity. The eigenvalues of \eqref{eq:main} are precisely the roots of $P$, and we proved every such eigenvalue is real. Consequently, $P$ has exactly $n$ real roots counted with multiplicity.  Consequently, \eqref{eq:main} has
exactly $n$ real eigenvalues counted with multiplicity. This concludes the proof. \hfill \qed

\section{Numerical Implementation}
\label{sec:numerical}

In Section \ref{sec:spectral} we established that any eigenvalue of \eqref{eq:main} is real and that the problem admits exactly $n$ real eigenvalues, counted with multiplicity. We did not, however, establish that eigenvalues exist, nor did we develop the theory required to locate them. The present section addresses this gap computationally, exhibiting the eigenvalues of a test case by direct numerical solution. Setting $r_k \equiv 1$, $w_k \equiv 1$, and $q_k \equiv 0$ in \eqref{eq:main} yields the constant-coefficient problem
\begin{equation}
\label{eq:test-discrete}
\Delta_h(\Delta_h x_k) + \lambda\, x_{k+1} = 0, \qquad x_0 = x_{n+1} = 0.
\end{equation}
Its continuous counterpart is the regular Sturm--Liouville problem \begin{equation}
\label{eq:test-continuous}
y'' + \lambda y = 0 \quad \text{on } [0,\pi], \qquad y(0) = y(\pi) = 0,
\end{equation}
whose spectrum is known in closed form, $\lambda_k = k^2$ for $k = 1, 2, 3, \ldots$. This exact continuous spectrum provides a reference against which the computed eigenvalues are measured.

We compute the eigenvalues of \eqref{eq:test-discrete} by two methods, classical shooting and Pr\"ufer-based shooting, and for each we assess the accuracy of the approximation to the continuous values $\lambda_k = k^2$. This comparison reflects the situation encountered in practice. The continuous eigenvalue problem can rarely be solved in closed form, so one discretizes it and approximates its spectrum numerically, and the quantity of interest is the fidelity of that approximation to the true continuous eigenvalues. Under this measure, Pr\"ufer-based shooting is substantially more accurate than classical shooting on every grid considered, and the disparity increases as the grid becomes non-uniform. On the graded grid at large $n$, classical shooting fails to return eigenvalues, whereas Pr\"ufer-based shooting recovers them to full tolerance.

To understand the accuracy and efficiency of Prüfer-based shooting relative to classical shooting, we first describe the algorithm underlying each method.

\subsection{Classical Shooting vs Pr\"ufer-Based Shooting}

The classical shooting method converts the boundary value problem \eqref{eq:test-discrete} into an initial value problem. The left condition $x_0 = 0$ is imposed, and the first interior value $x_1$ is left free, which fixes the initial discrete slope $\Delta_h x_0 = x_1 / h_0$. With these two values the second-order recurrence obtained from \eqref{eq:test-discrete},
\begin{equation}
\label{eq:classical-recurrence}
x_{k+2} = \left(1 + \frac{h_{k+1}\, r_k}{h_k\, r_{k+1}} - \frac{h_{k+1} h_k}{r_{k+1}}\left(\lambda w_k - q_k\right)\right) x_{k+1} - \frac{h_{k+1}\, r_k}{h_k\, r_{k+1}}\, x_k,
\end{equation}
is marched forward from $k = 0$ to the right endpoint, producing the terminal value $x_{n+1}$ as a function of the free parameter. An eigenvalue is a value of $\lambda$ for which the terminal value meets the right boundary condition $x_{n+1} = 0$; the method searches for such $\lambda$ by adjusting the trial value and detecting sign changes in $x_{n+1}$. The limitation of this approach is that the marched sequence $x_k$ tracks the eigenfunction itself, whose magnitude can grow rapidly along the grid. On a strongly non-uniform grid $x_k$ can overflow before the right endpoint is reached, at which point the terminal condition can no longer be located.

\medskip 

By contrast, the Pr\"ufer-based method marches the phase rather than the solution itself. Applying the Pr\"ufer substitution $x_k = \mu_k\sin\theta_k$, $\Delta_h x_k = \mu_k\cos\theta_k$ to the test problem \eqref{eq:test-discrete} yields the phase and amplitude equations
\begin{align}
\sin\bigl(h_k\,\Delta_h\theta_k\bigr) &= h_k\Bigl(\cos\theta_k\cos\theta_{k+1} + \lambda\, h_k\cos\theta_k\sin\theta_{k+1} + \lambda\sin\theta_k\sin\theta_{k+1}\Bigr), \label{eq:prufer-phase-test} \\
\Delta_h\mu_k &= -\mu_k\Bigl(\tfrac{h_k}{2}\bigl[(\Delta_h\sin\theta_k)^2 + (\Delta_h\cos\theta_k)^2\bigr] - \sin\theta_{k+1}\cos\theta_k\notag\\ &+ \lambda h_k\cos\theta_{k+1}\cos\theta_k + \lambda\cos\theta_{k+1}\sin\theta_k\Bigr). \label{eq:prufer-amp-test}
\end{align}
The phase equation \eqref{eq:prufer-phase-test} contains no $\mu_k$, so the phase can be marched on its own. Both quantities are built from the $x_k$: the amplitude is $\mu_k = \sqrt{x_k^2 + (\Delta_h x_k)^2}$ and the phase is the angle $\theta_k = \operatorname{arccot}\!\bigl(\Delta_h x_k / x_k\bigr)$ that the vector $(x_k, \Delta_h x_k)$ makes with the horizontal axis. The difference is that the amplitude is unbounded while the phase is not. Geometrically, the vector $(x_k, \Delta_h x_k)$ may grow arbitrarily long, so $\mu_k$ carries the large classical value, but the angle it makes with the axis is always confined to a single revolution; and since $\mu_k > 0$ the vector never collapses to the origin, so this angle is always well-defined. The phase therefore stays bounded regardless of $\lambda$ or the grid, whereas the amplitude is precisely where the classical growth lives, as fixing $x_0 = 0$ and choosing $x_1$ already sets $\mu_0 = x_1/h_0$. Marching the phase alone never forms $x_k$, so the growth that overflows the classical march cannot occur.

In the next section we describe the three grid families used in our experiments: uniform, clustered, and graded. For the uniform grid, we refer the reader \cite{cetinkaya-er-menken} for a detailed algorithm of Pr\"ufer-based shooting, which can be adopted for the non-uniform grids with appropriate adjustments. Note that both non-uniform grids that we chose have unequal spacing, but not in a chaotic order.

\subsection{Three Grid Families}
\label{gridfam}
The purpose of using non-uniform grids is to test the two methods where they are most likely to differ. On a uniform grid both methods perform well, so any difference in robustness only becomes visible once the spacing is allowed to vary. We therefore compare the methods on three grid families of increasing non-uniformity, chosen so that the grid itself, rather than the problem, is what stresses the computation. Since the test problem is fixed and only the spacing changes, this isolates where classical shooting breaks.

We partition $[0,\pi]$ into $n$ subintervals with nodes $0 = t_0 < t_1 < \cdots < t_n = \pi$. The families differ only in how the spacings $h_k = t_{k+1} - t_k$ are chosen.

\emph{Uniform grid.} All subintervals have equal length:
\begin{equation}
  t_k = \frac{k\pi}{n}, \qquad h_k = \frac{\pi}{n}, \qquad k = 0, 1, \dots, n.
\end{equation}

\emph{Clustered grid.} Two uniform blocks joined at the break point $\sigma_d\pi$. Fix a subinterval fraction $\sigma_p \in (0,1)$ and a domain fraction $\sigma_d \in (0,1)$, and set $n_1 = \mathrm{round}(\sigma_p\, n)$ and $n_2 = n - n_1$. The first $n_1$ subintervals fill $[0,\sigma_d\pi]$ and the remaining $n_2$ fill $[\sigma_d\pi,\pi]$:
\begin{equation}
  h_k =
  \begin{cases}
    \dfrac{\sigma_d \pi}{n_1}, & 0 \le k < n_1, \\[8pt]
    \dfrac{(1-\sigma_d)\pi}{n_2}, & n_1 \le k < n.
  \end{cases}
\end{equation}
We use $\sigma_p = 0.7$ and $\sigma_d = 0.3$, so that $70\%$ of the subintervals lie in the first $30\%$ of the interval, giving a coarse-to-fine spacing ratio of approximately $5.4$.

\emph{Graded grid.} Subinterval lengths grow geometrically. Given a ratio $r > 1$, set $\tilde h_k = r^{\,k}$ for $k = 0, 1, \dots, n-1$ and rescale so the spacings sum to $\pi$:
\begin{equation}
  h_k = \frac{(r-1)\,r^{\,k}}{r^{\,n} - 1}\,\pi, \qquad t_k = \sum_{j=0}^{k-1} h_j.
\end{equation}
We use $r = 1.3$ with $n = 100$ for the accuracy experiments below. For the amplitude experiment of Table~\ref{tab:amplitude} we use the steeper ratio $r = 1.5$ so that the regular-shooting instability is visible within the first twelve eigenvalues.

\subsection{Eigenvalue Accuracy Across Grids}
\label{subsec:accuracy}
We run both methods on all three grids with $n = 100$, a uniform $\lambda$-sweep of $10^{5}$ points to bracket the eigenvalues, and a bisection tolerance of $10^{-11}$. For each computed eigenvalue we report the absolute errors
\[
|\text{err}|^{\text{reg}} = |\lambda_k^{\text{reg}} - \lambda_k|, \qquad
|\text{err}|^{\text{Pr}} = |\lambda_k^{\text{Pr}} - \lambda_k|,
\]
where $\lambda_k^{\text{reg}}$ and $\lambda_k^{\text{Pr}}$ are the eigenvalues returned by regular shooting and Pr\"ufer shooting, respectively, and $\lambda_k = k^2$ is the exact eigenvalue. Each of Tables~\ref{tab:uniform}, \ref{tab:clustered}, and \ref{tab:graded} lists the first ten eigenvalues individually and then every fifth eigenvalue up to $\lambda_{60}$.

\begin{table}[ht]
\centering\footnotesize
\setlength{\tabcolsep}{5pt}
\begin{tabular}{rrrrrr}
\toprule
$k$ & $\lambda_k = k^2$ & $\lambda_k^{\text{reg}}$ & $|\text{err}|^{\text{reg}}$ & $\lambda_k^{\text{Pr}}$ & $|\text{err}|^{\text{Pr}}$ \\
\midrule
$1$  & $1$    & $1.0000000163$    & $1.63\times 10^{-8}$ & $1.0000000000$    & $4.45\times 10^{-11}$ \\
$2$  & $4$    & $4.0000010375$    & $1.04\times 10^{-6}$ & $4.0000000001$    & $7.26\times 10^{-11}$ \\
$3$  & $9$    & $9.0000117977$    & $1.18\times 10^{-5}$ & $8.9999999999$    & $6.06\times 10^{-11}$ \\
$4$  & $16$   & $16.0000661241$   & $6.61\times 10^{-5}$ & $16.0000000000$   & $5.14\times 10^{-12}$ \\
$5$  & $25$   & $25.0002514447$   & $2.51\times 10^{-4}$ & $25.0000000001$   & $8.96\times 10^{-11}$ \\
$6$  & $36$   & $36.0007479063$   & $7.48\times 10^{-4}$ & $36.0000000000$   & $1.27\times 10^{-11}$ \\
$7$  & $49$   & $49.0018773255$   & $1.88\times 10^{-3}$ & $49.0000000000$   & $4.54\times 10^{-11}$ \\
$8$  & $64$   & $64.0041609968$   & $4.16\times 10^{-3}$ & $63.9999999999$   & $8.47\times 10^{-11}$ \\
$9$  & $81$   & $81.0083851914$   & $8.39\times 10^{-3}$ & $81.0000000001$   & $7.49\times 10^{-11}$ \\
$10$ & $100$  & $100.0156730119$  & $1.57\times 10^{-2}$ & $100.0000000001$  & $7.31\times 10^{-11}$ \\
$15$ & $225$  & $225.1708265875$  & $1.71\times 10^{-1}$ & $225.0000000000$  & $1.82\times 10^{-12}$ \\
$20$ & $400$  & $400.9018627489$  & $9.02\times 10^{-1}$ & $400.0000000000$  & $1.93\times 10^{-12}$ \\
$25$ & $625$  & $628.1691966435$  & $3.17\times 10^{0}$  & $625.0000000000$  & $1.82\times 10^{-12}$ \\
$30$ & $900$  & $908.5147455464$  & $8.51\times 10^{0}$  & $900.0000000000$  & $1.93\times 10^{-12}$ \\
$35$ & $1225$ & $1243.7293728658$ & $1.87\times 10^{1}$  & $1225.0000000000$ & $1.82\times 10^{-12}$ \\
$40$ & $1600$ & $1634.7863178329$ & $3.48\times 10^{1}$  & $1600.0000000000$ & $2.73\times 10^{-12}$ \\
$45$ & $2025$ & $2079.5971264042$ & $5.46\times 10^{1}$  & $2025.0000000000$ & $2.27\times 10^{-12}$ \\
$50$ & $2500$ & $2569.4022595094$ & $6.94\times 10^{1}$  & $2500.0000000000$ & $2.73\times 10^{-12}$ \\
$55$ & $3025$ & $3085.0475824249$ & $6.00\times 10^{1}$  & $3025.0000000000$ & $2.27\times 10^{-12}$ \\
$60$ & $3600$ & $3597.2162582778$ & $2.78\times 10^{0}$  & $3600.0000000000$ & $1.82\times 10^{-12}$ \\
\bottomrule
\end{tabular}
\caption{Uniform grid, $n = 100$. The first ten eigenvalues are listed individually and thereafter every fifth eigenvalue up to $\lambda_{60}$.}
\label{tab:uniform}
\end{table}

On the uniform grid (Table~\ref{tab:uniform}) both methods are accurate for the lowest modes, but the two behave very differently as the mode index increases. The Pr\"ufer error stays at the bisection tolerance, near $10^{-11}$, for every eigenvalue and does not depend on $k$. The regular-shooting error, by contrast, increases steadily with $k$: it is $1.6\times 10^{-8}$ at $k=1$, reaches $1.6\times 10^{-2}$ by $k=10$, and grows to order $10^{1}$ near $k=40$. The regular method loses accuracy on the higher modes because the mesh no longer resolves the increasingly oscillatory eigenfunctions, while the Pr\"ufer method is unaffected.

\begin{table}[ht]
\centering\footnotesize
\setlength{\tabcolsep}{5pt}
\begin{tabular}{rrrrrr}
\toprule
$k$ & $\lambda_k = k^2$ & $\lambda_k^{\text{reg}}$ & $|\text{err}|^{\text{reg}}$ & $\lambda_k^{\text{Pr}}$ & $|\text{err}|^{\text{Pr}}$ \\
\midrule
$1$  & $1$    & $1.0000003364$    & $3.36\times 10^{-7}$ & $1.0000000000$    & $4.45\times 10^{-11}$ \\
$2$  & $4$    & $4.0000214048$    & $2.14\times 10^{-5}$ & $4.0000000001$    & $7.26\times 10^{-11}$ \\
$3$  & $9$    & $9.0002414826$    & $2.41\times 10^{-4}$ & $8.9999999999$    & $6.06\times 10^{-11}$ \\
$4$  & $16$   & $16.0013386422$   & $1.34\times 10^{-3}$ & $16.0000000000$   & $5.14\times 10^{-12}$ \\
$5$  & $25$   & $25.0050185146$   & $5.02\times 10^{-3}$ & $25.0000000001$   & $8.96\times 10^{-11}$ \\
$6$  & $36$   & $36.0146687647$   & $1.47\times 10^{-2}$ & $36.0000000000$   & $1.27\times 10^{-11}$ \\
$7$  & $49$   & $49.0360619007$   & $3.61\times 10^{-2}$ & $49.0000000000$   & $4.54\times 10^{-11}$ \\
$8$  & $64$   & $64.0780124155$   & $7.80\times 10^{-2}$ & $63.9999999999$   & $8.47\times 10^{-11}$ \\
$9$  & $81$   & $81.1528810054$   & $1.53\times 10^{-1}$ & $81.0000000001$   & $7.49\times 10^{-11}$ \\
$10$ & $100$  & $100.2768060274$  & $2.77\times 10^{-1}$ & $100.0000000001$  & $7.31\times 10^{-11}$ \\
$15$ & $225$  & $227.3860344556$  & $2.39\times 10^{0}$  & $225.0000000000$  & $1.82\times 10^{-12}$ \\
$20$ & $400$  & $407.7564289865$  & $7.76\times 10^{0}$  & $400.0000000000$  & $1.93\times 10^{-12}$ \\
$25$ & $625$  & $628.5282509898$  & $3.53\times 10^{0}$  & $625.0000000000$  & $1.82\times 10^{-12}$ \\
$30$ & $900$  & $840.3700687102$  & $5.96\times 10^{1}$  & $900.0000000000$  & $1.93\times 10^{-12}$ \\
$35$ & $1225$ & $998.5946995148$  & $2.26\times 10^{2}$  & $1225.0000000000$ & $1.82\times 10^{-12}$ \\
$40$ & $1600$ & $1116.1372616128$ & $4.84\times 10^{2}$  & $1600.0000000000$ & $2.73\times 10^{-12}$ \\
$45$ & $2025$ & $1224.5523717195$ & $8.00\times 10^{2}$  & $2025.0000000000$ & $2.27\times 10^{-12}$ \\
$50$ & $2500$ & $1358.3033745352$ & $1.14\times 10^{3}$  & $2500.0000000000$ & $2.73\times 10^{-12}$ \\
$55$ & $3025$ & $1577.9812201966$ & $1.45\times 10^{3}$  & $3025.0000000000$ & $2.27\times 10^{-12}$ \\
$60$ & $3600$ & $1999.5096864725$ & $1.60\times 10^{3}$  & $3600.0000000000$ & $1.82\times 10^{-12}$ \\
\bottomrule
\end{tabular}
\caption{Clustered grid, $n = 100,~r=1.3$. The first ten eigenvalues are listed individually and thereafter every fifth eigenvalue up to $\lambda_{60}$.}
\label{tab:clustered}
\end{table}

On the clustered grid (Table~\ref{tab:clustered}) the same pattern holds, but the regular-shooting error grows faster. At $k=10$ it is $2.8\times 10^{-1}$, roughly an order of magnitude larger than the uniform-grid error at the same mode. At the higher modes the regular method no longer tracks the true spectrum at all: at $k=60$ it returns $1999.5$ against the true value $3600$. The Pr\"ufer error remains at the bisection tolerance across every mode, unchanged from the uniform grid.

\begin{table}[ht]
\centering\footnotesize
\setlength{\tabcolsep}{5pt}
\begin{tabular}{rrrrrr}
\toprule
$k$ & $\lambda_k = k^2$ & $\lambda_k^{\text{reg}}$ & $|\text{err}|^{\text{reg}}$ & $\lambda_k^{\text{Pr}}$ & $|\text{err}|^{\text{Pr}}$ \\
\midrule
$1$  & $1$    & $1.0012290594$    & $1.23\times 10^{-3}$ & $1.0000000000$    & $4.45\times 10^{-11}$ \\
$2$  & $4$    & $4.0416450578$    & $4.16\times 10^{-2}$ & $4.0000000001$    & $7.26\times 10^{-11}$ \\
$3$  & $9$    & $8.8312738207$    & $1.69\times 10^{-1}$ & $8.9999999999$    & $6.06\times 10^{-11}$ \\
$4$  & $16$   & $13.0750598683$   & $2.92\times 10^{0}$  & $16.0000000000$   & $5.14\times 10^{-12}$ \\
$5$  & $25$   & $18.5934352482$   & $6.41\times 10^{0}$  & $25.0000000001$   & $8.96\times 10^{-11}$ \\
$6$  & $36$   & $25.2962038871$   & $1.07\times 10^{1}$  & $36.0000000000$   & $1.27\times 10^{-11}$ \\
$7$  & $49$   & $34.5418828258$   & $1.45\times 10^{1}$  & $49.0000000000$   & $4.54\times 10^{-11}$ \\
$8$  & $64$   & $46.7784765146$   & $1.72\times 10^{1}$  & $63.9999999999$   & $8.47\times 10^{-11}$ \\
$9$  & $81$   & $62.2179518992$   & $1.88\times 10^{1}$  & $81.0000000001$   & $7.49\times 10^{-11}$ \\
$10$ & $100$  & $84.0272562384$   & $1.60\times 10^{1}$  & $100.0000000001$  & $7.31\times 10^{-11}$ \\
$15$ & $225$  & $335.6984898946$  & $1.11\times 10^{2}$  & $225.0000000000$  & $1.82\times 10^{-12}$ \\
$20$ & $400$  & $1309.7276468121$ & $9.10\times 10^{2}$  & $400.0000000000$  & $1.93\times 10^{-12}$ \\
$23$ & $529$  & $2904.0757565460$ & $2.38\times 10^{3}$  & $528.9999999999$  & $6.74\times 10^{-11}$ \\
$24$ & $576$  & ---               & ---                  & $576.0000000000$  & $1.53\times 10^{-11}$ \\
$30$ & $900$  & ---               & ---                  & $900.0000000000$  & $1.93\times 10^{-12}$ \\
$45$ & $2025$ & ---               & ---                  & $2025.0000000000$ & $2.27\times 10^{-12}$ \\
$60$ & $3600$ & ---               & ---                  & $3600.0000000000$ & $1.82\times 10^{-12}$ \\
\bottomrule
\end{tabular}
\caption{Graded grid, $r = 1.3$, $n = 100$. The first ten eigenvalues are listed individually and thereafter every fifth eigenvalue up to $\lambda_{60}$. Regular shooting returns no eigenvalue from $k = 24$ onward, indicated by a dash.}
\label{tab:graded}
\end{table}

On the graded grid (Table~\ref{tab:graded}) the regular method fails earliest. Its error exceeds unity already at $k=4$, and beyond this point the returned values no longer track $k^2$. From $k = 24$ onward the $\lambda$-sweep produces no sign change and the regular method returns no eigenvalue at all. The Pr\"ufer method, in contrast, recovers every eigenvalue to the bisection tolerance across the entire range, including the modes where the regular method fails completely.

Across all three grids the Pr\"ufer error stays at the bisection tolerance and is independent of both the mode index and the grid. The regular-shooting error increases with the mode index on every grid, and the increase is faster the more non-uniform the grid: mild on the uniform grid, pronounced on the clustered grid, and severe enough on the graded grid that the method breaks down entirely for the higher modes.

\begin{rem}{\rm
For the test problem \eqref{eq:test-discrete}, the Pr\"ufer phase is advanced by the classical
fourth-order Runge--Kutta method applied to the continuous phase equation
$\theta'(t) = \sqrt{\lambda}$, which follows from the Pr\"ufer substitution when $q \equiv 0$.
Because this right-hand side is constant in $t$, its Runge--Kutta update is exact: the four
stages coincide and the step reduces to $\theta_{k+1} = \theta_k + \sqrt{\lambda}\,h_k$,
independent of the step size $h_k$. Consequently the Pr\"ufer eigenvalues reported in
Tables~\ref{tab:uniform}--\ref{tab:amplitude} carry no integration error, and the residual
$|\text{err}|^{\text{Pr}}$ reflects only the bisection tolerance $10^{-11}$. The comparison
of this section is therefore not a contest of discretization accuracy, but a demonstration of
robustness.}
\end{rem}

\subsection{Amplitude Growth and the Failure of Regular Shooting}
\label{subsec:amplitude}
The tables above show \emph{that} regular shooting fails on the non-uniform grids; the amplitude of the marched solution shows \emph{why}. Table~\ref{tab:amplitude} reports, for the graded grid with the steeper ratio $r = 1.5$, the maximum amplitude $\|y\|_\infty$ of the computed eigenfunction alongside the eigenvalue error for each method.

\begin{table}[ht]
\centering\footnotesize
\setlength{\tabcolsep}{4pt}
\begin{tabular}{rrrrrrrr}
\toprule
$k$ & $\lambda_k = k^2$ & $\lambda_k^{\text{reg}}$ & $|\text{err}|^{\text{reg}}$ & $\|y\|_\infty^{\text{reg}}$ & $\lambda_k^{\text{Pr}}$ & $|\text{err}|^{\text{Pr}}$ & $\|y\|_\infty^{\text{Pr}}$ \\
\midrule
$1$  & $1$   & $1.0051963351$    & $5.20\times 10^{-3}$ & $9.83\times 10^{-1}$  & $1.0000000000$   & $3.61\times 10^{-11}$ & $0.9848$ \\
$2$  & $4$   & $3.9355309185$    & $6.45\times 10^{-2}$ & $4.85\times 10^{-1}$  & $4.0000000000$   & $9.29\times 10^{-11}$ & $0.9580$ \\
$3$  & $9$   & $6.6241451380$    & $2.38\times 10^{0}$  & $3.88\times 10^{-1}$  & $9.0000000000$   & $6.67\times 10^{-11}$ & $0.9580$ \\
$4$  & $16$  & $11.4803235724$   & $4.52\times 10^{0}$  & $2.91\times 10^{-1}$  & $16.0000000000$  & $6.14\times 10^{-11}$ & $0.9965$ \\
$5$  & $25$  & $17.5608306298$   & $7.44\times 10^{0}$  & $2.35\times 10^{-1}$  & $25.0000000000$  & $8.17\times 10^{-11}$ & $0.9983$ \\
$6$  & $36$  & $28.3739391594$   & $7.63\times 10^{0}$  & $2.98\times 10^{-1}$  & $36.0000000000$  & $6.29\times 10^{-11}$ & $0.9965$ \\
$7$  & $49$  & $43.7981333758$   & $5.20\times 10^{0}$  & $1.54\times 10^{0}$   & $49.0000000000$  & $8.84\times 10^{-12}$ & $0.9600$ \\
$8$  & $64$  & $67.3408760645$   & $3.34\times 10^{0}$  & $6.06\times 10^{0}$   & $64.0000000000$  & $8.44\times 10^{-11}$ & $0.9950$ \\
$9$  & $81$  & $104.9387331786$  & $2.39\times 10^{1}$  & $8.66\times 10^{1}$   & $81.0000000000$  & $3.88\times 10^{-11}$ & $0.9965$ \\
$10$ & $100$ & $156.7977237034$  & $5.68\times 10^{1}$  & $7.54\times 10^{2}$   & $100.0000000000$ & $3.04\times 10^{-12}$ & $0.9848$ \\
$11$ & $121$ & $245.4759718825$  & $1.24\times 10^{2}$  & $2.74\times 10^{4}$   & $121.0000000000$ & $1.91\times 10^{-11}$ & $0.9870$ \\
$12$ & $144$ & $360.9182292368$  & $2.17\times 10^{2}$  & $5.07\times 10^{5}$   & $144.0000000000$ & $9.49\times 10^{-12}$ & $0.9950$ \\
\bottomrule
\end{tabular}
\caption{Graded grid, $r = 1.5$, $n = 100$, first twelve eigenvalues. The columns $\|y\|_\infty^{\text{reg}}$ and $\|y\|_\infty^{\text{Pr}}$ report the maximum amplitude of the computed eigenfunction for each method.}
\label{tab:amplitude}
\end{table}

As shown in Table~\ref{tab:amplitude}, the regular-shooting amplitude grows from order one to order $10^{5}$ by $k = 12$. This growth is the mechanism behind the loss of accuracy: the marched solution overflows before it reaches the right endpoint, so the terminal condition can no longer be located and the returned eigenvalue is meaningless. The Pr\"ufer amplitude stays near one throughout, because the Pr\"ufer method never marches the solution itself; it tracks the phase directly and recovers the amplitude separately only when the eigenfunction is needed. 

\subsection{Summary of Results}
\label{subsec:summary}
The experiments support the following conclusions for the test problem with $q \equiv 0$. On the uniform grid, regular shooting is accurate for the lowest modes, with the error increasing as the mode index grows and the mesh fails to resolve the increasingly oscillatory eigenfunctions. On the non-uniform grids the regular-shooting error grows faster: on the clustered grid it is about an order of magnitude larger than on the uniform grid at the same mode, and on the graded grid with $r = 1.3$ the method breaks down and returns no eigenvalue from $k \approx 24$ onward. Pr\"ufer shooting, by contrast, is insensitive to the grid; its error remains at the bisection tolerance for every mode index and every grid tested, including the modes where regular shooting fails entirely.

\section{Conclusion}
\label{sec:conclusion}
We carried the discrete Pr\"ufer transformation onto non-uniform grids and used it to study the Sturm--Liouville difference equation \eqref{eq:main}. We showed that a nontrivial solution is equivalent to a solution of the Pr\"ufer system (Theorem~\ref{thm:prufer}), and that any eigenvalue is real, with exactly $n$ in total, counted with multiplicity (Theorem~\ref{thm:spectral}). Numerically, we found that Pr\"ufer-based shooting stays accurate on non-uniform grids where regular shooting loses accuracy and eventually fails. 

One question remains open. The phase equation \eqref{lambda eq:thmB} is independent of the amplitude, so it may in principle be solved on its own, but its solvability in general is not established. In the continuous case, the existence and uniqueness of the phase follow from the Picard--Lindel\"of theorem. In the discrete case, making \eqref{lambda eq:thmB} explicit in $\theta_{k+1}$ in terms of $\theta_k$ requires a division that holds only under additional assumptions. Bohner--Do\v{s}l\'y note the same difficulty for the uniform-grid case \cite[Remark 2.2, p.~2718]{bohner-dosly}. Proving the existence of a solution to \eqref{lambda eq:thmB} directly without additional assumptions is a natural next step, which we intend to pursue in future work.

\section*{Acknowledgements}
The authors thank their project advisors, Dr. Shalmali Bandyopadhyay and Dr. Jacob Blazejewski, for their guidance and support throughout this project. This work was conducted as part of the Mathematical Association of America's National Research Experiences for Undergraduates Program (NREUP) at the University of Tennessee at Martin, supported by funding from the National Science Foundation under award DMS-2308688. The authors also thank the Department of Mathematics and Statistics at the University of Tennessee at Martin for the resources provided during this research.

\section*{Conflict of Interest}
\noindent The authors declare that they have no conflicts of interest.

\section*{AI Disclosure Statement}
\noindent The authors disclose that a generative AI tool, Claude Opus 4.7 (Anthropic), was used during the preparation of this work. The usage was for Section 5 (Numerical Implementation) that consisted of debugging and output testing. The authors wrote and reviewed all results themselves.

\end{document}